\documentclass[12pt]{amsart}
\usepackage{amscd,amssymb,amsmath,multicol,float,scalefnt,cancel,array,stmaryrd,graphicx,tikz}

\newtheorem{thm}[equation]{Theorem}
\numberwithin{equation}{section}
\newtheorem{cor}[equation]{Corollary}

\newtheorem{defin}[equation]{Definition}

\newtheorem{prop}[equation]{Proposition}

\newtheorem{fig}[equation]{Figure}

\begin{document}
\raggedbottom \voffset=-.7truein \hoffset=0truein \vsize=8truein
\hsize=6truein \textheight=8truein \textwidth=6truein
\baselineskip=18truept

\def\mapright#1{\ \smash{\mathop{\longrightarrow}\limits^{#1}}\ }
\def\mapleft#1{\smash{\mathop{\longleftarrow}\limits^{#1}}}
\def\mapup#1{\Big\uparrow\rlap{$\vcenter {\hbox {$#1$}}$}}
\def\mapdown#1{\Big\downarrow\rlap{$\vcenter {\hbox {$\ssize{#1}$}}$}}
\def\mapne#1{\nearrow\rlap{$\vcenter {\hbox {$#1$}}$}}
\def\mapse#1{\searrow\rlap{$\vcenter {\hbox {$\ssize{#1}$}}$}}
\def\mapr#1{\smash{\mathop{\rightarrow}\limits^{#1}}}
\def\ss{\smallskip}
\def\vp{v_1^{-1}\pi}
\def\at{{\widetilde\alpha}}
\def\sm{\wedge}
\def\la{\langle}
\def\ra{\rangle}
\def\ev{\text{ev}}
\def\od{\text{od}}
\def\on{\operatorname}
\def\ol#1{\overline{#1}{}}
\def\spin{\on{Spin}}
\def\cat{\on{cat}}
\def\lbar{\ell}
\def\qed{\quad\rule{8pt}{8pt}\bigskip}
\def\ssize{\scriptstyle}
\def\a{\alpha}
\def\bz{{\Bbb Z}}
\def\Rhat{\hat{R}}
\def\im{\on{im}}
\def\ct{\widetilde{C}}
\def\ext{\on{Ext}}
\def\sq{\on{Sq}}
\def\eps{\epsilon}
\def\ar#1{\stackrel {#1}{\rightarrow}}
\def\br{{\bold R}}
\def\bC{{\bold C}}
\def\bA{{\bold A}}
\def\bB{{\bold B}}
\def\bD{{\bold D}}
\def\bh{{\bold H}}
\def\bQ{{\bold Q}}
\def\bP{{\bold P}}
\def\bx{{\bold x}}
\def\bo{{\bold{bo}}}
\def\si{\sigma}
\def\Vbar{{\overline V}}
\def\dbar{{\overline d}}
\def\wbar{{\overline w}}
\def\Sum{\sum}
\def\tfrac{\textstyle\frac}
\def\tb{\textstyle\binom}
\def\Si{\Sigma}
\def\w{\wedge}
\def\equ{\begin{equation}}
\def\b{\beta}
\def\G{\Gamma}
\def\L{\Lambda}
\def\g{\gamma}
\def\k{\kappa}
\def\ahat{\widehat{a}}
\def\psit{\widetilde{\Psi}}
\def\tht{\widetilde{\Theta}}
\def\psiu{{\underline{\Psi}}}
\def\thu{{\underline{\Theta}}}
\def\aee{A_{\text{ee}}}
\def\aeo{A_{\text{eo}}}
\def\aoo{A_{\text{oo}}}
\def\aoe{A_{\text{oe}}}
\def\vbar{{\overline v}}
\def\endeq{\end{equation}}
\def\sn{S^{2n+1}}
\def\zp{\bold Z_p}
\def\cR{{\mathcal R}}
\def\cS{{\mathcal S}}
\def\cD{{\mathcal D}}
\def\cj{{\cal J}}
\def\zt{{\bold Z}_2}
\def\bs{{\bold s}}
\def\bof{{\bold f}}
\def\bq{{\bold Q}}
\def\be{{\bold e}}
\def\Hom{\on{Hom}}
\def\ker{\on{ker}}
\def\kot{\widetilde{KO}}
\def\coker{\on{coker}}
\def\da{\downarrow}
\def\colim{\operatornamewithlimits{colim}}
\def\zphat{\bz_2^\wedge}
\def\io{\iota}
\def\om{\omega}
\def\Prod{\prod}
\def\e{{\cal E}}
\def\zlt{\Z_{(2)}}
\def\exp{\on{exp}}
\def\abar{{\overline a}}
\def\xbar{{\overline x}}
\def\ybar{{\overline y}}
\def\zbar{{\overline z}}
\def\mbar{{\overline m}}
\def\nbar{{\overline n}}
\def\sbar{{\overline s}}
\def\kbar{{\overline k}}
\def\bbar{{\overline b}}
\def\et{{\widetilde E}}
\def\ni{\noindent}
\def\tsum{\textstyle \sum}
\def\coef{\on{coef}}
\def\den{\on{den}}
\def\lcm{\on{l.c.m.}}
\def\vi{v_1^{-1}}
\def\ot{\otimes}
\def\psibar{{\overline\psi}}
\def\thbar{{\overline\theta}}
\def\mhat{{\hat m}}
\def\exc{\on{exc}}
\def\ms{\medskip}
\def\ehat{{\hat e}}
\def\etao{{\eta_{\text{od}}}}
\def\etae{{\eta_{\text{ev}}}}
\def\dirlim{\operatornamewithlimits{dirlim}}
\def\gt{\widetilde{L}}
\def\lt{\widetilde{\lambda}}
\def\St{\widetilde{S}}
\def\ft{\widetilde{f}}
\def\sgd{\on{sgd}}
\def\lfl{\lfloor}
\def\rfl{\rfloor}
\def\ord{\on{ord}}
\def\gd{{\on{gd}}}
\def\rk{{{\on{rk}}_2}}
\def\nbar{{\overline{n}}}
\def\MC{\on{MC}}
\def\lg{{\on{lg}}}
\def\cH{\mathcal{H}}
\def\cR{\mathcal{R}}
\def\cS{\mathcal{S}}
\def\cE{\mathcal{E}}
\def\cP{\mathcal{P}}
\def\N{{\Bbb N}}
\def\Z{{\Bbb Z}}
\def\Q{{\Bbb Q}}
\def\R{{\Bbb R}}
\def\C{{\Bbb C}}
\def\l{\left}
\def\r{\right}
\def\interior{\on{int}}
\def\mo{\on{mod}}
\def\xt{\times}
\def\notimm{\not\subseteq}
\def\Remark{\noindent{\it  Remark}}
\def\kut{\widetilde{KU}}
\def\newd{\mathrm{d}}
\def\GC{\mathrm{GC}}
\def\*#1{\mathbf{#1}}
\def\0{$\*0$}
\def\1{$\*1$}
\def\22{$(\*2,\*2)$}
\def\33{$(\*3,\*3)$}
\def\ss{\smallskip}
\def\ssum{\sum\limits}
\def\dsum{\displaystyle\sum}
\def\la{\langle}
\def\ra{\rangle}
\def\on{\operatorname}
\def\od{\text{od}}
\def\ev{\text{ev}}
\def\o{\on{o}}
\def\U{\on{U}}
\def\lg{\on{lg}}
\def\a{\alpha}
\def\bz{{\Bbb Z}}
\def\eps{\varepsilon}
\def\bc{{\bold C}}
\def\bN{{\bold N}}
\def\nut{\widetilde{\nu}}
\def\tfrac{\textstyle\frac}
\def\b{\beta}
\def\G{\Gamma}
\def\g{\gamma}
\def\zt{{\Bbb Z}_2}
\def\zth{{\bold Z}_2^\wedge}
\def\by{{\bold y}}
\def\bx{{\bold x}}
\def\bof{{\bold f}}
\def\bq{{\bold Q}}
\def\be{{\bold e}}
\def\lline{\rule{.6in}{.6pt}}
\def\xb{{\overline x}}
\def\xbar{{\overline x}}
\def\ybar{{\overline y}}
\def\zbar{{\overline z}}
\def\ebar{{\overline \be}}
\def\nbar{{\overline n}}
\def\ubar{{\overline u}}
\def\bbar{{\overline b}}
\def\et{{\widetilde e}}
\def\lf{\lfloor}
\def\rf{\rfloor}
\def\ni{\noindent}
\def\ms{\medskip}
\def\xhat{{\widehat x}}
\def\what{{\widehat w}}
\def\Yhat{{\widehat Y}}
\def\abar{{\overline{a}}}
\def\minp{\min\nolimits'}
\def\mul{\on{mul}}
\def\N{{\Bbb N}}
\def\Z{{\Bbb Z}}
\def\S{\Sigma}
\def\Q{{\Bbb Q}}
\def\R{{\Bbb R}}
\def\C{{\Bbb C}}
\def\notint{\cancel\cap}
\def\cS{\mathcal S}
\def\cRt{\widehat{\mathcal R}}
\def\cR{\mathcal R}
\def\el{\ell}
\def\TC{\on{TC}}
\def\wgt{\on{wgt}}
\def\wpt{\widetilde{p_2}}
\def\dstyle{\displaystyle}
\def\Om{\Omega}
\def\ds{\dstyle}
\def\tz{tikzpicture}
\def\sx{{[\![6]\!]}}
\def\sxk{{[\![6]\!]-\{k\}}}
\def\zcl{\on{zcl}}
\def\Vb#1{{\overline{V_{#1}}}}
\title
{The geodesic complexity of $n$-dimensional Klein bottles}
\author{Donald M. Davis}
\address{Department of Mathematics, Lehigh University\\Bethlehem, PA 18015, USA}
\email{dmd1@lehigh.edu}
\author{David Recio-Mitter}
\address{Department of Mathematics, Lehigh University\\Bethlehem, PA 18015, USA}
\email{dar318@lehigh.edu}
\date{December 16, 2019}


\maketitle
\begin{abstract} The geodesic complexity of a metric space $X$ is the smallest $k$ for which there is a partition of $X \times X$ into ENRs $E_0,\ldots,E_k$ on each of which there is a
continuous choice  of minimal geodesic $\sigma(x_0,x_1)$ from $x_0$ to $x_1$. We prove that the geodesic complexity of an $n$-dimensional Klein bottle equals $2n$. The topological complexity of $K_n$ remains unknown for $n>2$.
 \end{abstract}

\section{Introduction}


The motion planning problem is of central importance in the field of robotics. The geodesic complexity $\GC(X)$ of a metric space $(X,\newd)$ is a measure of the minimal instability of any optimal motion planner on $X$ (optimal in the sense that the motions are always along shortest paths). It was introduced recently by the second author in \cite{RM19}, inspired by Farber's topological complexity $\TC(X)$ \cite{F03}, which is a measure of the minimal instability over all motion planners on $X$, not necessarily along shortest paths. In fact, $\TC(X)$ is defined for all topological spaces and is a homotopy invariant, while $\GC(X)$ depends on the metric \cite{RM19}.

Let $PX\to X\times X$ denote the free path fibration, which maps each path $\gamma$ in $X$ to the pair $(\gamma(0),\gamma(1))$. Furthermore, let $GX \subset PX$ consist of the minimal geodesics. This is paths whose length equals the distance between the endpoints: $\ell(\gamma)=d(\gamma(0),\gamma(1))$. The restriction of the free path fibration to $GX$ defines a map $\pi:GX\to X\times X$, which is not a fibration in general.

\begin{defin}[\cite{F03}]
The {\em topological complexity} of a space $X$, $\TC(X)$, is defined to be the smallest $k$ for which there exists a decomposition into  $k+1$ disjoint ENRs $X \times X = \bigsqcup_{i=0}^k E_i$ such that there are local sections $s_i \colon E_i \to PX$ of $PX\to X\times X$.
\end{defin}

\begin{defin}[\cite{RM19}]\label{gcdef}
The {\em geodesic complexity} of a metric space $(X, \newd)$, $\GC(X, \newd)$, is defined to be the smallest $k$ for which there exists a decomposition into $k+1$ disjoint ENRs $X \times X = \bigsqcup_{i=0}^k E_i$ such that there are local sections $s_i \colon E_i \to GX$ of $\pi$.
\end{defin}

It was shown in \cite{RM19} that these numbers are different in general, even though they agree in many cases.

The higher Klein bottles $K_n$ were introduced by the first author in the course of his work on planar polygon spaces \cite{D}:

\[K_n = (S^1)^n / ( z_1 , \ldots , z_{n-1} , z_n ) \sim ( \bar z_1 , \ldots , \bar z_{n-1} , -z_n ) .\]

These spaces are a generalization of the standard Klein bottle, which corresponds to $K_2$. We consider the metric space $(K_n,\newd)$ with the flat metric coming from the universal covering $\mathbb{R}^n$.

The topological complexity of the Klein bottle $K_2$ was an open problem for over a decade, even after the topological complexity had been determined for all other orientable and nonorientable surfaces. In 2017 $\TC(K_2)$ was finally computed by Cohen and Vandembroucq \cite{CV} and later by Iwase, Sakai and Tsutaya \cite{IST}. The geodesic complexity $\GC(K_2)$ (with the flat metric) is equal to the topological complexity $\TC(K_2)$, as was shown by the second author in \cite{RM19}.

The topological complexity of $K_n$ is currently unknown for $n\ge3$, and a proof seems to be out of reach at this point. The geodesic complexity is computed in this article:

\begin{thm}\label{mainthm}
The geodesic complexity of the flat higher Klein bottles $K_n$ is given by \[\GC(K_n)=2n.\]
\end{thm}

Most of the work in the proof consists in understanding the \textit{total cut locus} of $K_n$. Briefly, the total cut locus of a metric space $X$ consists of all pairs of points $(x,y)$ in $X\times X$ such that there are at least two shortest paths between $x$ and $y$. In other words, it is the union over all $x$ in $X$ of the cut loci of each $x$, in the usual sense.

It is conceivable that in fact $\TC(K_n)=\GC(K_n)$ for all $n$  could be shown by a general argument, in which case this work would yield the values $\TC(K_n)$ as an application of the previous theorem.
\section{Upper bound}
In this section, we prove $\GC(K_n)\le 2n$ by demonstrating explicit geodesic motion planning rules, postponing many details to later sections.

Let $\sim$ be the equivalence relation on $\R^n$ generated by $x\sim x+e_i$, $1\le i\le n-1$, where $e_i$ is the unit vector in the $i$th coordinate, and
\begin{equation}\label{rel}(x_1,\ldots,x_{n-1},x_n)\sim(1-x_1,\ldots,1-x_{n-1},x_n+1).\end{equation}
The higher Klein bottle $K_n$ is the quotient space, with $p:\R^n\to K_n$ the quotient map. The metric $d_K$ on $K_n$ is defined by
$$d_K(y,y')=\min\{d(x,x'):x\in p^{-1}(y), x'\in p^{-1}(y')\},$$
where $d$ is the Euclidean metric on $\R^n$.
If $P,Q\in\R^n$, we say that $Q$ is {\it $K$-close} to $P$ if $d_K(p(P),p(Q))=d(P,Q)$. If $\sigma_{P,Q}$ denotes the uniform linear path from $P$ to $Q$, then
the geodesics in $K_n$  are paths of the form $p\circ\sigma_{P,Q}$ such that $Q$ is $K$-close to $P$.

For $P=(x_1,\ldots,x_n)\in I^n$, let $C(P)$ consist of the following points of
 $p^{-1}(p(P))-\{P\}$ which are, in a certain sense, closest to $P$. The set $C(P)$ contains the points $P\pm e_1,\ldots,P\pm e_{n-1}$ and also the points $P\pm2e_n$ and $(y_1,\ldots,y_{n-1},x_n\pm1)$, where
\begin{equation}\label{CP}y_i\in\begin{cases}\{-x_i,1-x_i\}&0<x_i<\frac12\\
\{x_i\}&x_i\in\{0,\frac12,1\}\\
\{1-x_i,2-x_i\}&\frac12<x_i<1.\end{cases}\end{equation}
Let $\cR(P)$ denote the intersection of the closed half-spaces containing $P$ bounded by the perpendicular bisectors of the segments from $P$ to each of the points of $C(P)$. Note that the polytope $\cR(P)$ consists of those points $Q$ which are $K$-close to $P$, and that
$(p\times p)(\{(P,Q):P\in[0,1)^n, Q\in\partial(\cR(P))\}$ is the total cut locus of $K_n$.

 For $0\le j\le n$, let $R_j(P)$ denote the set of interiors of $j$-dimensional faces of $\cR(P)$. In particular, the set $R_n(P)$ consists of the single set $\interior(\cR(P))$. The equivalence relation $\sim$ induces one on each set $R_j(P)$. We consider subsets $\cD:=D_1\times\cdots\times D_n\subset\R^n$, where
\begin{equation}\label{D}D_i=\begin{cases}(0,\frac12)\cup(\frac12,1)\text{ or }\{0,\frac12\}&1\le i\le n-1\\
(0,1)\text{ or }\{0\}&i=n.\end{cases}\end{equation}
Note that for each $\cD$, $p:\cD\to p(\cD)$ is a homeomorphism.

Most of our work goes into proving the following theorem.
\begin{thm}\label{thm1} For $n\ge2$, the sets $\cD$ listed above can be partitioned into finitely many subsets $M_\a$, which are analytic spaces of varying dimensions, on which the polytopes $\cR(P)$ vary continuously and bijectively, preserving $\sim$. That is, if we choose a point $P_0\in M_\a$, there exist polytope-equivalences $\theta_P:\cR(P_0)\to \cR(P)$ preserving $\sim$  for all $P\in M_\a$, which vary continuously with $P$ (i.e., if $x\in\cR(P_0)$, then $P\mapsto \theta_P(x)$ is continuous). Also, $\theta_{P_0}=1$.

Moreover, (a) if $\dim(M_\a)\le\dim(M_{\a'})$ and $M_\a\ne M_{\a'}$, then the closure of $M_\a$ is disjoint from $M_{\a'}$ or any set equivalent to it, and (b) if $F$ is a $j$-face in $\cR(P_0)$ for $P_0\in M_\a$ and $\la P_i\ra\to P$ for $P_i\in M_\a$ and $P\in I^n$, then $\lim\theta_{P_i}(F)$ is contained in a $j'$-face of $\cR(P)$ for $j'\le j$.
\end{thm}

By {\it analytic space} we mean an open subset of a real variety (the set of solutions of a set of polynomial (for us, quadratic) equations over $\R$). Our analytic spaces may possibly have singularities, but they have dimension, like a manifold.
The following corollary is half of Theorem \ref{mainthm}.
\begin{cor} \label{cor}$\GC(K_n)\le 2n$.\end{cor}
\begin{proof} Fix $M_\a$ and $P_0\in M_\a$. Choose a representative of each equivalence class of $R_j(P_0)$, and let $R'_{j}(P_0)$ denote their union. Let $R'_{j}(P)=\theta_P(R'_{j}(P_0))$, and
$$\cS_{\a,j}=\{(P,Q):P\in M_\a,\ Q\in R'_{j}(P)\}.$$
The bijective function $(p\times p)|\cS_{\a,j}$ has a continuous inverse. [\![It suffices to show that there does not exist a sequence $(P_i,Q_i)\in\cS_{\a,j}$ converging to a point $(P,Q)$ not in $\cS_{\a,j}$ which is equivalent ($\sim$) to a point of $\cS_{\a,j}$. Since $M_\a\subset\cD$ for some $\cD$ as above, we must have $P\in M_\a$. Using the functions $\theta_{P_i}$ and $\theta_P$, we may assume that the convergence $Q_i\to Q$ is taking place in $\cR(P_0)$, where disjointness of interiors of $j$-faces, together with the fact that equivalence of faces preserves dimensions, implies that the contemplated situation cannot occur.]\!]
The function $\cS_{\a,j}\to (\R^n)^I$ defined by $(P,Q)\mapsto \sigma_{P,Q}$ is continuous, hence so is the composite
$$(p\times p)(\cS_{\a,j})\mapright{(p\times p)^{-1}}\cS_{\a,j}\to (\R^n)^I\to (K_n)^I,$$
giving a continuous choice on $(p\times p)(\cS_{\a,j})$ of geodesics in $K_n$. This is called a geodesic motion planning rule on $(p\times p)(\cS_{\a,j})$.
Using all $\a$ and all $j$, the sets $(p\times p)(\cS_{\a,j})$ give a partition of $K_n\times K_n$ into subsets on which continuous choice of geodesics can be found.
However some of the $\cS_{\a,j}$'s can be combined:

We claim that if $\dim(M_\a)+j=\dim(M_{\a'})+j'$, then $(p\times p)(\cS_{\a,j})$ and $(p\times p)(\cS_{\a',j'})$ are topologically disjoint (or equal). This implies that, for $0\le i\le 2n$, we can use as our sets  the union of all $(p\times p)(\cS_{\a,j})$ for which $\dim(M_\a)+j=i$, establishing the corollary, since we have $2n+1$ subsets admitting continuous choices of geodesics.

To prove the claim, suppose $(p(P_i),p(Q_i))$ converges to $(p(P),p(Q))$, where $P_i\in M_\a$, $Q_i\in \theta_{P_i}(R'_{j}(P_0))$, $P\in M_{\a'}$, and $Q\in \theta_P(R'_{j'}(P_0))$.
Then a subsequence, which we also call $\la P_i\ra$, converges to a point equivalent to $P$. Then (a) implies that either $M_{\a'}=M_\a$ or $\dim(M_{\a'})<\dim(M_\a)$. In the first case, then $j'=j$, and so $\cS_{\a,j}=\cS_{\a',j'}$. In the second case, $j'>j$. Since $R_j(P_0)$ is finite, we may assume that there is a $j$-face $F$ of $\cR(P_0)$ such that $Q_i\in\theta_{P_i}(F)$ for all $i$, and since the number of equivalence classes under $p$ is finite, then we may assume a sequence of $Q_i$'s converges, so by (b), $j'\le j$, a contradiction.

 By separating $x_n\in\{0,1\}$ from $x_n\in(0,1)$, we remove any concern about continuity when $x_n=0$, and similarly for $x_i\in\{0,\frac12,1\}$.
\end{proof}

The sets $M_\a$ are quite simple for $n\le 6$, as described in the following theorem, which is a consequence of Theorem \ref{verthm}. See especially material immediately following Figure \ref{fig4}. Here we initiate the practice, continued throughout, of using $(a_1,\ldots,a_n)$ for the coordinates of $P$, the point away from which we will be moving, while $(x_1,\ldots,x_n)$ will be used for coordinates in $\cR(P)$, the points toward which we move from $P$.

\begin{thm} \label{thm2} If $n\le 4$, Theorem \ref{thm1} holds using as $M_\a$ exactly the sets $\cD$ described prior to the theorem. For $n=5$, it holds with the one change that
$\cD=((0,\frac12)\cup(\frac12,1))^{4}\times D_5$ be replaced by $\cE=\{\frac14,\frac34\}^4\times D_5$ and $\cD-\cE$. If $n=6$, $\cD=((0,\frac12)\cup(\frac12,1))^5\times D_6$ must be replaced by
$$\cE=\{(a_1,\ldots,a_{5}):\dsum_{i=1}^{5}\min((a_i-\tfrac14)^2,(a_i-\tfrac34)^2)=\tfrac{1}{16}\}\times D_6$$ and $\cD-\cE$, and also  $\cD=((0,\frac12)\cup(\frac12,1))^4\times \{0,\frac12\}\times D_6$ must be replaced by $\cE=\{\frac14,\frac34\}^4\times\{0,\frac12\}\times D_6$ and $\cD-\cE$, and similarly for permutations of the first five factors.\end{thm}

\section{Examples when $n=2$ or $3$}\label{sec2}

In this section, we illustrate the sets $\cR(P)$ when $n=2$ or 3, when we can actually picture them. For $0<a\le\frac12$, let
\begin{equation}\label{Del}\Delta_a=a-2a^2=\tfrac18-2(a-\tfrac14)^2.\end{equation}
This formula is very important throughout the paper.

We begin with the case $n=2$. Let $P=(a_1,a_2)$ with $0<a_1<\frac12$. For $\eps\in\{0,1\}$, the equation of the perpendicular bisector of the segment between $(a_1,a_2)$ and $(\eps-a_1,a_2+1)$ is $x_2=a_2+\frac12+(2a_1-\eps)(x_1-\frac12\eps)$.
It is easy to check that $\cR(a_1,a_2)$ is a hexagon as pictured in Figure \ref{expl}, with $V_\pm=(\frac12-a_1,a_2\pm(\frac12+\Delta_{a_1}))$ and $C_{\pm,\pm'}=(a_1\pm\frac12,a_2\pm'(\frac12-\Delta_{a_1}))$. We have $V_+\sim C_{--}\sim C_{+-}$, and $V_-\sim C_{-+}\sim C_{++}$. Also $V_+C_{-+}\sim C_{--}V_-$ and $V_+C_{++}\sim C_{+-}V_-$ and  $C_{-+}C_{--}\sim C_{++}C_{+-}$. For $\frac12<a_1<1$, the shape is the same with formulas modified slightly.
\begin{fig}\label{expl}
{\bf Polygon when $n=2$}
\begin{center}
\begin{\tz}[scale=.75]
\draw (0,0) -- (4,0) -- (4,4) -- (0,4) -- (0,0);
\draw [red] (1,3.5) -- (3,2.5) -- (3,-.5) -- (1,-1.5) -- (-1,-.5) -- (-1,2.5) -- (1,3.5);
\node at (-1,5) [circle,fill,inner sep=1pt]{};
\node at (1,1)  [circle,fill,inner sep=1pt]{};
\node at (3,5) [circle,fill,inner sep=1pt]{};
\node at (5,1)  [circle,fill,inner sep=1pt]{};
\node at (3,-3)  [circle,fill,inner sep=1pt]{};
\node at (-1,-3)  [circle,fill,inner sep=1pt]{};
\node at (-3,1)  [circle,fill,inner sep=1pt]{};
\node at (3,2.5) {$C_{++}$};
\node at (1,3.5) {$V_+$};
\node at (1,-1.5) {$V_-$};
\node at (-1,2.5) {$C_{-+}$};
\node at (-1,-.5) {$C_{--}$};
\node at (3,-.5) {$C_{+-}$};
\node at (1.3,1) {$P$};
\end{\tz}
\end{center}
\end{fig}

If $a_1\in\{0,\frac12,1\}$, then $\cR(a_1,a_2)$ is a unit square centered at $(a_1,a_2)$, with vertical sides equivalent, horizontal sides equivalent (in the opposite order), and all four vertices equivalent. The continuous  bijective variation of $\cR(P)$, preserving $\sim$, for $P=(a_1,a_2)$ in any one of the domains $(0,\frac12)\times(0,1)$, $(\frac12,1)\times(0,1)$ (hence in $((0,\frac12)\cup(\frac12,1))\times(0,1)$), $\{0,\frac12\}\times(0,1)$, $(0,1)\times\{0\}$, etc., is clear. As an example of a motion planning rule, we might choose, for $P_0\in M_\a=(0,\frac12)\times(0,1)$, the interiors of the segments $C_{-+}V_+$, $V_+C_{++}$, and $C_{++}C_{+-}$ as representatives of equivalence classes of $R_1(P_0)$, with $R'_{1}(P_0)$ being their union, and choose $R_0'(P_0)=\{V_+,C_{++}\}$. The motion planning rule for $\{(P,Q):P\in M_\a, Q\in \theta_P(R'_{j}(P_0))\}$ is $(p(P),p(Q))\mapsto (p\times p)(\sigma_{P,Q})$, $j\in\{0,1,2\}$.

Now we consider the case $n=3$.  Consider first the points $P=(a_1,a_2,a_3)$ with $0<a_1,a_2<\frac12$. The planes bisecting the lines from $P$ to $P\pm e_1$ and $P\pm e_2$ yield a region bounded by four vertical walls, rising above and below the square in the $x_1x_2$-plane with vertices at $(a_1\pm\frac12,a_2\pm' \frac12)$. These walls will be capped above and below by  pyramids with vertices $V_\pm=(\frac12-a_1,\frac12-a_2,a_3\pm(\frac12+\Delta_{a_1}+\Delta_{a_2}))$.
The top vertex $V_+$ is the intersection of four planes with equations, for $\eps_i\in\{0,1\}$, $$x_3=a_3+\tfrac12+\dsum_{i=1}^2(2a_i-\eps_i)(x_i-\tfrac12\eps_i).$$
Each of these planes comes down and intersects two of the walls. In Figure \ref{expl2}, we depict schematically part of one of these regions $\cR(P)$. The letters $F$ and $S$ denote a front and side wall.

\begin{fig}\label{expl2}
{\bf Polytope when $n=3$}
\begin{center}
\begin{\tz}[scale=.45]
\draw (0,0) -- (.6,1.2) -- (3,4) -- (3,6) -- (1.2,7.6) -- (0,10) -- (-.6,8.8) -- (-3,6) -- (-3,4) -- (-1.2,2.4) -- (0,0);
\draw (.6,1.2) -- (0,4) -- (-1.2,2.4);
\draw (-.6,8.8) -- (0,6) -- (1.2,7.6);
\draw (0,4) -- (0,6);
\node at (-1.2,2.4) [circle,fill,inner sep=1.2pt]{};
\node at (.6,1.2)  [circle,fill,inner sep=1.2pt]{};
\node at (-.6,8.8) [circle,fill,inner sep=1.2pt]{};
\node at (1.2,7.6)  [circle,fill,inner sep=1.2pt]{};
\node at (3.7,6) {$C_{+}$};
\node at (0,10.4) {$V_+$};
\node at (0,-.4) {$V_-$};
\node at (3.7,4) {$C_{-}$};
\node at (-1.5,5) {$F$};
\node at (1.5,5) {$S$};
\end{\tz}
\end{center}
\end{fig}

The points at the $C_\pm$ level are at height $a_3\pm(\frac12-\Delta_{a_1}-\Delta_{a_2})$. Since $\Delta_a\le\frac18$, the sides of the upper pyramid do not intersect those of the lower pyramid. For $n\ge6$, such an intersection complicates the analysis.
The points indicated by $\bullet$s have coordinates $(\frac12-a_1,a_2\pm\frac12,a_3\pm'(\frac12+\Delta_{a_1}-\Delta_{a_2}))$ or $(a_1\pm\frac12,\frac12-a_2,a_3\pm'(\frac12-\Delta_{a_1}+\Delta_{a_2}))$. Points on the walls are equivalent to points on the opposite walls. Points on the slanted quadrilaterals are equivalent to points on the vertically displaced quadrilateral, but in opposite directions in both the $x_1$- and $x_2$-directions. The continuous dependence of this region $\cR(P)$ on $P$ is clear from these formulas. If $a_1$ or $a_2$ (or both) is in $(\frac12,1)$, the description is similar, with $\frac12-a_i$ replaced by $\frac32-a_i$, and $\Delta_a=\frac18-2(a-\frac34)^2$. The change at $a_i=\frac12$ is not a problem for continuity since $\frac12$ has been removed from the domain.

If $a_1\in\{0,\frac12\}$, then $\cR(P)=[a_1-\frac12,a_1+\frac12]\times\cR(a_2,a_3)$, where $\cR(a_2,a_3)$ is the 2-dimensional region described in the discussion for $n=2$, and $(a_1-\frac12,x_2,x_3)\sim(a_1+\frac12,x_2,x_3)$. To see this, we observe that the planes bounding $\cR(P)$ are $\R\times\mathcal{P}$, where $\mathcal{P}$ bounds $\cR(a_2,a_3)$, and $\{a_1\pm\frac12\}\times\R^2$. After reminding the reader that the case when $a_3=0$ (or 1) is separated from the case when $a_3\in(0,1)$, but is similar in nature, we conclude that we have completed the claimed description of the required regions when $n=3$.

\section{Vertices of a polytope $\cR(P)$} \label{sec3}

In this section, we determine the vertices of the polytopes $\cR(P)$ when $0<a_i<\frac12$ for $1\le i\le n-1$ and $a_n=0$. The restriction to $a_n=0$ is just to simplify formulas. Just add $a_n$ to all $x_n$ values for the general case. We let $\cRt(P)$ denote an approximation of $\cR(P)$ which does not take into account truncating due to the points $P\pm2e_n$ of $C(P)$.

\begin{defin}\label{def}Let $P=(a_1,\ldots,a_{n-1},0)$ with $0<a_1,\ldots,a_{n-1}<\frac12$. Let $\cRt(P)$ denote the polytope in $\R^n$ which is the intersection of the following half-spaces:
\begin{eqnarray}&&x_i\le a_i+\tfrac12,\  i=1,\ldots,n-1\label{30}\\
&&x_i\ge a_i-\tfrac12,\ i=1,\ldots,n-1\label{31}\\
&&x_n\le\tfrac12+\ds\sum_{i=1}^{n-1}(2a_i-\delta_i)(x_i-\tfrac12\delta_i),\ \delta_i\in\{0,1\}\label{*}\\
&&x_n\ge-\tfrac12-\ds\sum_{i=1}^{n-1}(2a_i-\delta_i)(x_i-\tfrac12\delta_i),\ \delta_i\in\{0,1\}.\label{neg}
\end{eqnarray}\end{defin}
Intuitively, $\cRt(P)$ can be thought of as walls capped by a pyramid on top and another on the bottom. The vertices will occur where a $j$-face of a pyramid
intersects an $(n-j)$-dimensional wall.
However, when $n\ge 6$, some parts of the top pyramid might intersect a wall at a negative value of $x_n$. Such a vertex will be outside the lower pyramid; it will not satisfy (\ref{neg}). This vertex, and a corresponding vertex on the bottom pyramid, will be replaced by a number of ``middle'' vertices on the intersection, at $x_n=0$, of the lower and upper pyramids.

When $n\ge6$, some vertices of $\cRt(P)$ will  lie above the hyperplane $x_n=1$, which is the perpendicular bisector of the segment connecting $P$ to $P+2e_n$. The desired polytope $\cR(P)$ is the intersection of $\cRt(P)$ with the region $-1\le x_n\le 1$.

We define a {\it labeled set} $S$ to be a set $S$ together with a function $\eps_S:S\to\{0,1\}$, $i\mapsto\eps_i$. Also, $|S|$ is the cardinality of $S$, and $[\![n-1]\!]=\{1,\ldots,n-1\}$.
With $\Delta$ as in (\ref{Del}), let $\Delta(S)=\dsum_{i\in S}\Delta_{a_i}$, and let $\St=[\![n-1]\!]-S$.
The quantity \begin{equation}\label{Kdef}K(S)=\tfrac12-\Delta(S)+\Delta(\St)\end{equation} will play a central role in our results.
Note that, for fixed $S$, $K(S)$ is a quadratic function of $a_1,\ldots,a_{n-1}$. It depends only on the set $S$, not the labeling. It satisfies
\begin{equation}\label{KSp} K(S)+K(\St)=1,\end{equation} which will be useful later.

The vertices of $\cR(P)$ are described as follows.
\begin{thm}\label{verthm}
 For $P=(a_1,\ldots,a_{n-1},0)$ with $0<a_1,\ldots,a_{n-1}<\frac12$, $\cR(P)$ has vertices of possibly three types.
\begin{itemize}
\item[$\bullet$] For labeled sets $S\subset[\![n-1]\!]$ satisfying $0\le K(S)\le1$, there are ``standard'' vertices $v_S^\pm$ with
\begin{equation}\label{36}x_i=\begin{cases}a_i-\tfrac12+\eps_i&i\in S\\
\tfrac12-a_i&i\in\St\\
\pm K(S)&i=n.\end{cases}\end{equation}
     If $K(S)=0$, then $v_S^+=v_S^-$.
\item[$\bullet$] For labeled sets $S$ satisfying $K(S)<0< K(S-\{k\})$, there are ``middle'' vertices $v^0_{S,k}$ with
\begin{equation}\label{xi}x_i=\begin{cases}a_i-\tfrac12+\eps_i&i\in S-\{k\}\\
\tfrac12-a_i&i\in\St\\
a_k-\tfrac12+\eps_k-\tfrac{K(S)}{2a_k-\eps_k}&i=k\\
0&i=n.\end{cases}\end{equation}

\item[$\bullet$] For labeled sets $S\cup\{k\}$ satisfying $K(S\cup\{k\})< 1<K(S)$, there are ``truncating'' vertices $v^\pm_{S,k}$ satisfying
\begin{equation}\label{trun}x_i=\begin{cases} a_i-\tfrac12+\eps_i&i\in S\\
\tfrac12-a_i&i\in\St-\{k\}\\

a_k-\tfrac12+\eps_k+\tfrac{1-K(S\cup\{k\})}{2a_k-\eps_k}&i=k\\
\pm1&i=n.\end{cases}\end{equation}\end{itemize}
There are no other vertices of $\cR(P)$.
\end{thm}
Note that by (\ref{KSp}), $\cR(P)$ has truncating vertices for $\St$ iff it has middle vertices for $S$. Also note the notation: truncating vertices are distinguished from standard vertices by having a second subscript.
\begin{proof} Let $t=|\St|$. For $\cRt(P)$, we seek solutions of the inequalities of Definition \ref{def} satisfying $x_i=a_i-\frac12+\eps_i$ for $i\in S$ which have equality in $t+1$ independent equations of type (\ref{*}) and (\ref{neg}).
The following fact is very useful. For $\eps\in\{0,1\}$ and $0<a<\frac12$,
$$(2a-\eps)(x-\tfrac12\eps)=\begin{cases}\Delta_a&x=\tfrac12-a\\
-\Delta_a&x=a-\tfrac12+\eps\\
-\Delta_a+2a&\eps=0,\ x=a+\tfrac12\\
-\Delta_a+1-2a&\eps=1,\ x=a-\tfrac12.\end{cases}$$
Because of this, if, for $i\in S$,  a pair of inequalities (\ref{*}) (for $\delta_i=0$ and 1) both satisfy the inequality and one has equality, then it must be the case that $\delta_i=\eps_i$ and $(2a_i-\delta_i)(x_i-\frac12\delta_i)=\Delta_{a_i}$.
After setting $x_i=a_i-\frac12+\eps_i$ for $i\in S$, the relevant inequalities of Definition \ref{def} become
\begin{eqnarray}\label{one}x_n&\le&\tfrac12-\Delta(S)+\dsum_{i\in\St}(2a_i-\delta_i)(x_i-\tfrac12\delta_i)\\
x_n&\ge&-\tfrac12+\Delta(S)\label{two}-\dsum_{i\in\St}(2a_i-\delta_i)(x_i-\tfrac12\delta_i).\end{eqnarray}

Note that $v_S^+$ satisfies equality for all the inequalities of (\ref{one}), and it satisfies the inequalities (\ref{two}) iff $K(S)\ge0$. Thus, assuming $K(S)\ge0$, $v_S^+$ is a vertex of $\cRt(P)$, and similarly so is $v_S^-$. There can be no other vertices associated to $S$ with $t+1$ independent equations of type (\ref{one}) because $t+1$ linearly independent equations in $t+1$ variables can have at most one solution. Note that the system of equations associated with (\ref{one}) has rank $t+1$, as it is easily seen to reduce to the equations $x_i=\frac12-a_i$, $i\in\St$, and $x_n=\frac12-\Delta(S)+\sum_{\St}2a_ix_i$.

 Next we  consider the truncation. If $K(S\cup\{k\})<1<K(S)$, then $v^+_{S}$ lies above the hyperplane $x_n=1$, as does $v_T^+$ for any $T\subset S$. We obtain new vertices as the intersection of the edge between $v_{S\cup\{k\}}^+$ and $v_{S}^+$ with the hyperplane $x_n=1$. Note that for each labeled set $S$, there are two labeled sets $S\cup\{k\}$, corresponding to the two possible labels on $k$. The points $v_{S\cup\{k\}}^+$ and $v_{S}^+$ differ only in the $k$th and $n$th components, which are $(a_k-\frac12+\eps_k,K(S\cup\{k\}))$ for $v_{S\cup\{k\}}^+$, and $(\frac12-a_k,K(S\cup\{k\})+2\Delta_{a_k})$ for $v^+_{S}$. The point on the segment between them with $x_n=1$ has
\begin{equation}\label{tr}x_k=a_k-\tfrac12+\eps_k+\frac{1-2a_k-\eps_k}{2\Delta_{a_k}}(1-K(S\cup\{k\})),\end{equation}
which simplifies as claimed.

The middle vertices are obtained similarly by computing the point where the segment from $v_S^+$ to $v_{S-\{k\}}^+$ meets the hyperplane $x_n=0$. Since $k$ is in the labeled set $S$, there are two vertices $v_{S,k}^0$ connected to each vertex $v_{S-\{k\}}^+$. As explained in the paragraph after Definition \ref{def}, these occur when the intersection of the slant hyperplanes associated to (\ref{*}) with the walls associated to (\ref{30}) and (\ref{31}) corresponding to the labeled set $S$ meet at a negative value $K(S)$ of $x_n$. The slant hyperplanes here have $\delta_i=\eps_i$ for $i\in S$, so (\ref{one}) and (\ref{two}) apply.

The vertex $v_S^+$ satisfies the wall equations for the labeled set $S$ and (\ref{one}) for $\delta_i=0$ or 1 when $i\in\St$, but does not satisfy (\ref{two}). Choose any $k$ having $K(S-\{k\})>0$, and replace the wall hyperplane $x_k=a_k-\frac12+\eps_k$ by any hyperplane (\ref{two}). Comparing this equation with the corresponding equation in (\ref{one}) yields $x_n=0$. The new $x_k$ value is nicely found by linearity similarly to (\ref{tr}). Alternatively, one can verify that (\ref{xi}) satisfies equality in all (\ref{*}).

The diagram in Figure \ref{fig3} is quite representative. The horizontal axis is $x_k$ and the vertical axis is $x_n$. The vertical lines represent the walls $x_k=a_k\pm\frac12$, and the horizontal line represents the hyperplane $x_n=0$. The ``vertices'' $v_S^+$ (resp.~$v_S^-$) are not on the polytope since they lie outside the hyperplanes coming up from $v_{S-\{k\}}^-$ (resp.~down from $v_{S-\{k\}}^+$).

\begin{minipage}{6in}
\begin{fig}\label{fig3}
{\bf Formation of middle vertices}
\begin{center}
\begin{\tz}[scale=.5]
\draw (-5,-4) -- (1,5.6);
\draw (-1,5.6) -- (5,-4);
\draw (5,4) -- (-1,-5.6);
\draw (1,-5.6) -- (-5,4);
\draw (-5,0) -- (5,0);
\draw (-4,-4) -- (-4,4);
\draw (4,-4) -- (4,4);
\node at (1.5,4.1) {$v_{S-\{k\}}^+$};
\node at (1.5,-4.1) {$v_{S-\{k\}}^-$};
\node at (4.5,-2) {$v_S^+$};
\node at (-4.5,-2) {$v_S^+$};
\node at (4.5,2) {$v_S^-$};
\node at (-4.5,2) {$v_S^-$};
\node at (-1.5,.5) {$v_{S,k}^0$};
\node at (1.5,.5) {$v_{S,k}^0$};
\node at (6.2,0) {$x_n=0$};
\node at (0,4) [circle,fill,inner sep=1.2pt]{};
\node at (-2.5,0)  [circle,fill,inner sep=1.2pt]{};
\node at (0,-4) [circle,fill,inner sep=1.2pt]{};
\node at (2.5,0)  [circle,fill,inner sep=1.2pt]{};
\end{\tz}
\end{center}
\end{fig}
\end{minipage}

We now show that there are no additional vertices obtained as intersections of hyperplanes of type (\ref{one}) and type (\ref{two}). First note that any intersection of the two types of hyperplanes must occur at $x_n=0$. This follows since if for all vectors  $\delta=\la\delta_i\ra$
$$\tfrac12-\Delta(S)+\dsum_{\St}(2a_i-\delta_i)(x_i-\tfrac12\delta_i)\ge x_n$$
and
$$\tfrac12-\Delta(S)+\dsum_{\St}(2a_i-\delta_i)(x_i-\tfrac12\delta_i)\ge -x_n,$$
with equality for some $\delta$ of each type, then $x_n=0$.

We are looking for $t+1$ linearly independent hyperplanes associated to (\ref{one}) and (\ref{two}) intersecting at $x_n=0$. It is enough to consider that all but one of them are of type (\ref{one}). To see this, as just noted, one hyperplane of type (\ref{two}) forces $x_n=0$, which then implies that points $(x_1,\ldots,x_n)$ which satisfy any other equation of type (\ref{two}) also satisfy the corresponding equation of type (\ref{one}), and so all but one of the type-(\ref{two}) hyperplanes can be replaced by type-(\ref{one}) hyperplanes.

As already noted, if $K(S)<0$, the putative $v_S^+$ does not satisfy (\ref{two}), but if also
 $K(S-\{k\})>0$, there will be a vertex
$v_{S-\{k\}}^+$, and $v_{S,k}^0$ can be found as above. If $K(S-\{k,\ell\})>K(S-\{k\})>0>K(S)$, there will not be an additional vertex obtained by intersecting a segment from $v_S^+$ to $v^+_{S-\{k,\ell\}}$ with $x_n=0$ because it would lie on an edge from $v_{S,k}^0$ to
$$\begin{cases}v_{S,\ell}^0&\text{if }K(S-\{\ell\})>0\\
v_{S-\{\ell\},k}^0&\text{if }K(S-\{\ell\})<0.\end{cases}$$
These vertices all have $x_i=\frac12-a_i$ for $i\in\St$, and $x_i=a_i-\frac12+\eps_i$ for $i\in S-\{k,\ell\}$, and lie in the plane (for simplicity, let $\eps_k=\eps_\ell=0$)
$$x_n=\tfrac12+2a_kx_k+2a_\ell x_\ell -\Delta(S-\{k,\ell\})+\Delta(\St).$$
See Figure \ref{fig5}.
 \end{proof}

\begin{minipage}{6in}
 \begin{fig}\label{fig5}
 {\bf Face containing middle vertices}
\begin{center}
\begin{\tz}[scale=.7]
\draw (-1,1.5) -- (1,1.5) -- (2,3) -- (0,4) -- (-2,3) -- (-1,1.5);
\draw (6,1.5) -- (8,1.5) -- (9,3) -- (7,4) -- (6,1.5);
\draw [dotted,thick] (-1,1.5) -- (0,0) -- (1,1.5);
\draw [dotted,thick] (7,0) -- (8,1.5);
\draw [dotted,thick] (5.6,.5) -- (6,1.5);
\node at (0,-1) {$K(S-\{\ell\})>0$};
\node at (7,-1) {$K(S-\{\ell\})<0$};
\node at (0,4) [circle,fill,inner sep=1.2pt]{};
\node at (2,3) [circle,fill,inner sep=1.2pt]{};
\node at (1,1.5) [circle,fill,inner sep=1.2pt]{};
\node at (-1,1.5) [circle,fill,inner sep=1.2pt]{};
\node at (-2,3) [circle,fill,inner sep=1.2pt]{};
\node at (6,1.5) [circle,fill,inner sep=1.2pt]{};
\node at (8,1.5) [circle,fill,inner sep=1.2pt]{};
\node at (9,3) [circle,fill,inner sep=1.2pt]{};
\node at (7,4) [circle,fill,inner sep=1.2pt]{};
\draw (0,0) circle [radius=3pt];
\draw (7,0) circle [radius=3pt];
\draw (5.6,.5) circle [radius=3pt];
\node at (0,4.5) {$v^+_{S-\{k,\ell\}}$};
\node at (7,4.5) {$v^+_{S-\{k,\ell\}}$};
\node [right] at (1,1.5) {$v_{S,k}^0$};
\node [right] at (8,1.5) {$v_{S,k}^0$};
\node [right] at (2,3) {$v_{S-\{k\}}^+$};
\node [right] at (9,3) {$v_{S-\{k\}}^+$};
\node [right] at (0,0) {$v_S^+$};
\node [right] at (7,0) {$v_S^+$};
\node [left] at (-1,1.5) {$v_{S,\ell}^0$};
\node [left] at (-2,3) {$v_{S-\{\ell\}}^+$};
\node [left] at (5.6,.5) {$v_{S-\{\ell\}}^+$};
\node [left] at (6,1.5) {$v_{S-\{\ell\},k}^0$};
\end{\tz}
\end{center}
\end{fig}
\end{minipage}

The diagram in Figure \ref{fig4} of a fictitious polytope in $\R^3$ can be very useful in visualizing the middle vertices as they sit in the whole polytope. This polytope is fictitious in two ways. First, the actual polytopes associated to $K_3$ do not have middle vertices, and second, whenever middle vertices occur, then truncating vertices occur, too, but that is not the case in our diagram. This represents a polytope with $K(\{1\})>0$, $K(\{2\})>0$, and $K(\{1,2\})<0$. Subscripts on the numbers 1 and 2 refer to the labeling, $-$ if $\eps_i=0$, and $+$ if $\eps_i=1$. Thus, for example, $v_{1_+2_+,2}^0$ is the middle vertex $v_{S,2}^0$ with $S=\{1,2\}$ and $\eps_1=\eps_2=1$.  Heights (i.e., $x_3$ values) decrease as the distance from $v_\emptyset^+$ increases in the diagram. The bottom vertex $v_\emptyset^-$ is the point at $\infty$. All edges and faces are indicated. The labeling for faces is that, for example, $1_-$ means the wall $x_1=a_1-\frac12$, while $(1_-2_+)^\pm$ means the hyperplane $x_3=\pm(\frac12+2a_1x_1+(2a_2-1)(x_2-\frac12))$. The diagram shows very clearly how vertices are intersections of three independent hyperplanes, and how the intersection of a slant hyperplane with its negative gives an edge at height 0.

\begin{minipage}{6in}
\begin{fig}\label{fig4}
{\bf Fictitious polytope}
\begin{center}
\begin{\tz}[scale=.6]
\node at (0,0) {$v_{\emptyset}^+$};
\node at (3,0) {$v_{1_+}^+$};
\node at (0,3) {$v_{2_+}^+$};
\node at (-3,0) {$v_{1_-}^+$};
\node at (0,-3) {$v_{2_-}^+$};
\node at (4.5,1.5) {$v_{1_+2_+,2}^0$};
\node at (1.9,4.5) {$v_{1_+2_+,1}^0$};
\node at (-1.9,4.5) {$v_{1_-2_+,1}^0$};
\node at (-4.5,1.5) {$v_{1_-2_+,2}^0$};
\node at (-1.9,-4.5) {$v_{1_-2_-,1}^0$};
\node at (1.9,-4.5) {$v_{1_+2_-,1}^0$};
\node at (4.5,-1.5) {$v_{1_+2_-,2}^0$};
\node at (6,0) {$v_{1_+}^-$};
\node at (0,6) {$v_{2_+}^-$};
\node at (-6,0) {$v_{1_-}^-$};
\node at (0,-6) {$v_{2_-}^-$};
\draw (.4,5.6) -- (1.0,4.9);
\draw (1.8,4) -- (3.9,2);
\draw (4.8,1) -- (5.8,.3);
\draw (6.4,0) -- (8.4,0);
\draw (5.8,-.2) -- (4.8,-1.2);
\draw (4,-1.9) -- (2.1,-4.2);
\draw (1.3,-4.9) -- (.3,-5.7);
\draw (0,-6.5) -- (0,-7.5);
\draw (0,-.5) -- (0,-2.6);
\draw (0,.5) -- (0,2.5);
\draw (.4,-3.4) -- (1.4,-4);
\draw (.4,0) -- (2.45,0);
\draw (-2.6,0) -- (-.4,0);
\draw (3.4,.4) -- (4.1,1.1);
\draw (3.4,-.4) -- (4,-1);
\draw (.4,3.4) -- (1.4,4.1);
\draw (0,6.4) -- (0,7.4);
\draw (-.4,5.6) -- (-1.4,4.9);
\draw (-1.4,4) -- (-.3,3.4);
\draw (-2.2,4) -- (-4.1,1.8);
\node at (-4.5,-1.5) {$v_{1_-2_-,2}^0$};
\draw (-5.6,.4) -- (-4.8,1.1);
\draw (-6.5,0) -- (-7.5,0);
\draw (-4,1.1) -- (-3.1,.4);
\draw (-5.6,-.4) -- (-4.8,-1.1);
\draw (-3.6,-1.8) -- (-2,-3.8);
\draw (-1,-5) -- (-.2,-5.6);
\draw (-1.4,-4.2) -- (-.2,-3.3);
\draw (-4.1,-1.1) -- (-3.2,-.4);
\node [red] at (4.5,0) {$1_+$};
\node [red] at (-4.5,0) {$1_-$};
\node [red] at (0,4.5) {$2_+$};
\node [red] at (0,-4.5) {$2_-$};
\node [red] at (4.5,5) {$(1_+2_+)^-$};
\node [red] at (-4.5,5) {$(1_-2_+)^-$};
\node [red] at (4.5,-5) {$(1_+2_-)^-$};
\node [red] at (-4.5,-5) {$(1_-2_-)^-$};
\node [red] at (2,2) {$(1_+2_+)^+$};
\node [red] at (-2,2) {$(1_-2_+)^+$};
\node [red] at (2,-2) {$(1_+2_-)^+$};
\node [red] at (-1.9,-2) {$(1_-2_-)^+$};
\end{\tz}
\end{center}
\end{fig}
\end{minipage}

We illustrate Theorem \ref{verthm} by describing the vertices of $\cR(P)$ for $P\in (0,\frac12)^{n-1}\times\{0\}$ and $n\le6$. Since $0<\Delta_a\le\frac18$, $K(S)\ge0$ is satisfied if $|S|\le 4$. Thus for $n\le5$, all standard vertices exist and there are no middle or truncating vertices. The only slight deviation is that if $n=5$, $|S|=4$, and $a_1=a_2=a_3=a_4=\frac14$, then $x_n=0$, so $v_S^+=v_S^-$.

For $n=6$, $\cR(P)$ depends on  $Z:=\sum_{i=1}^5 (a_i-\frac14)^2$. If $Z>\frac1{16}$, then standard vertices $v_S^\pm$ exist for all labeled sets $S$. There are $2\cdot3^5$ vertices, since for $i\in[\![5]\!]$, either $\eps_i=0$ or 1 or $i\not\in S$. If $Z=\frac1{16}$, there are standard vertices for all $S$ except that if $|S|=5$, then $v_S^+=v_S^-$. If $Z<\frac1{16}$,  then $0<K(S)<1$ for $1\le|S|\le4$, while $K(S)>1$ if $|S|=0$, and $K(S)<0$ if $|S|=5$.  There are standard vertices $v_S^\pm$ for all $S$ with $1\le|S|\le4$, and middle vertices $v_{[\![5]\!],k}^0$ for all labelings of $[\![5]\!]$. There are truncating vertices $v^\pm_{\emptyset,k}$ for the ten labeled sets $\{k\}$.

Using (\ref{KSp}), Theorem \ref{verthm} shows that the types of vertices that occur in $\cR(P)$ depend on the sign of $K(S)$ for each $S\subset[\![n-1]\!]$, and, for a fixed vertex type, the coordinates of the vertices are continuous functions of the $a_i$'s. We now show that equivalence of vertices depends just on the sets $S$ (without regard for labeling) and whether $K(S)=0$.
\begin{prop}\label{equivthm} If $P\in(0,\frac12)^{n-1}\times I$, the only equivalences of vertices of $\cR(P)$ are
\begin{itemize}
\item[a.] If $S=S'$ as sets, then $v_S^+\sim v_{S'}^+$, $v_S^-\sim v_{S'}^-$, $v_{S,k}^0\sim v_{S',k}^0$, $v_{S,k}^+\sim v_{S',k}^+$, and $v_{S,k}^-\sim v_{S',k}^-$, and if also $K(S)=0$ or 1, then $v_S^+\sim v_{S'}^-$.
\item[b.] If $S'=\St$ as sets, then $v_S^+\sim v_{S'}^-$, and if also $K(S)=0$ or 1, then $v_S^+\sim v_{S'}^+$ and $v_S^-\sim v_{S'}^-$.
\item[c.] If $k\in S$ and $K(S)<0<K(S-\{k\})$, then $v_{S,k}^0\sim v_{\St,k}^\pm$ for any labelings.
\end{itemize}
\end{prop}
\begin{proof} The general rule, which is immediate from the definition, is that
vertices $(x_1,\ldots,x_n)$ and $(x'_1,\ldots,x'_n)$ of $\cR(P)$ are equivalent iff either  $x_i-x'_i\in\Z$ for $1\le i\le n-1$ and $x_n-x_n'\in2\Z$, or
$x_i+x'_i\in\Z$ for $1\le i\le n-1$ and $x_n-x'_n\in 2\Z+1$.
Recall that $v_S^+=(x_1,\ldots,x_{n-1},a_n+K(S))$ and $v_S^-=(x_1,\ldots,x_{n-1},a_n-K(S))$, and similarly for $v_{S'}^\pm$.
First note that $x_i-x'_i\in\Z$ for $i\le n-1$ iff the sets $S$ and $S'$ are equal, and $x_i+x'_i\in\Z$ for $i\le n-1$ iff $S'=\St$.
Next note that
$$K(S)-K(S')=\begin{cases}0&\text{$S=S'$ as sets}\\ 2K(S)-1&S'=\St,\end{cases}$$
while the difference of the $n$th components of $v_S^+$ and $v_{S'}^-$ is given by
$$K(S)+K(S')=\begin{cases}2K(S)&\text{$S=S'$  as sets}\\ 1&S'=\St.\end{cases}$$
The proposition for standard vertices now follows easily from the general rule noted above, as do the other equivalences in part (a).

Part (c) is true since the sum of the $x_i$ values for the two vertices is an integer for all $i<n$, and the $x_n$ values differ by 1. We show this for the $x_k$ values by noting that the sum of the relevant $x_k$ values from (\ref{xi}) and (\ref{trun}) is
\begin{eqnarray*} &&2a_k-1+2\eps_k+\frac{K(S-\{k\})-K(S)}{2a_k-\eps_k}\\
&=&-1+2\eps_k+\frac{2a_k(2a_k-\eps_k)+2\Delta_{a_k}}{2a_k-\eps_k}\\&=&-1+2\eps_k+\frac{2a_k(1-\eps_k)}{2a_k-\eps_k}\\
&=&\eps_k\end{eqnarray*}
for $\eps_k\in\{0,1\}$. Other possible equivalences involving middle vertices are eliminated by showing that under the hypotheses of (\ref{xi}), $x_k$ cannot equal $a_k-\frac12+\eps_k$ or $\frac12-a_k$.
\end{proof}
An analogous result holds for $((0,\frac12)\cup(\frac12,1))^{n-1}$.

\section{Proof of Theorem \ref{thm1}}\label{pfsec}

We begin by showing how $\cD=(0,\frac12)^{n-1}\times(0,1)$ can be partitioned into subsets $M_\a$ as claimed in Theorem \ref{thm1}. The partitioning is done by the quadric surfaces $K(S)=0$ for subsets $S$ of $[\![n-1]\!]$. The partitioning just depends on $S$ as a set; the labeling on $S$ is used in describing vertices. Each set $M_\a$ is uniquely determined by a specification of whether, for each set $S$, $K(S)$ is positive, negative, or 0. We have described the rather simple partitioning  when $n=6$ in Theorem \ref{thm2}. We will now discuss the more-typical case $n=7$, as a precursor to the general proof.

When $n=7$,
$$K(S)=\tfrac{5-|S|}4+\dsum_{i\in S}2(a_i-\tfrac14)^2-\dsum_{i\in\St}2(a_i-\tfrac14)^2.$$
Since $0\le(a_i-\frac14)^2<\frac1{16}$, $K(S)>0$ if $|S|\le4$, but $K(S)$ can be positive or negative if $|S|=5$ or 6.  Note that by (\ref{KSp}), $K(\St)>1$ iff $K(S)<0$, so truncating vertices associated to $\St$ will occur exactly when middle vertices associated to $S$ occur. If $K(S)<0<K(S-\{k\})$, there are $2^{|S|}$ middle vertices $v_{S,k}^0$ and $2^{|\St|+2}$ truncating vertices $v_{\St,k}^\pm$, taking the labeling into account.
Domains in which one or more $K(S)$ equal 0 will have dimension less than 7, and will bound other regions.
Recall that the $n$th coordinate of points in $\cD$ does not play an important role in this part of the analysis. The description of the vertices of $\cR(P)$ assumed $a_n=0$; for arbitrary $a_n$, just add that amount to the $n$th coordinate of all vertices described in Theorem \ref{verthm}. The domain $M_\a$ will be a product with $(0,1)$ as its last factor.

We now describe the domains $M_\a$ in $(0,\frac12)^6$ when $n=7$, omitting the 7th factor. Let $b_i=a_i-\frac14$, so  $b_i\in(-\frac14,\frac14)$, and $\frac12K(S)=\frac{5-|S|}8+\sum_Sb_i^2-\sum_{\St}b_i^2$.
Let $Z=\sum_{i=1}^6b_i^2$. In the following list, $K(S)>0$ unless mentioned to the contrary. The regions $M_\a$ are of the following six types. For each, we tell (a) conditions on $b_i$, (b) conditions on $K(S)$, (c) types of vertices, and (d) dimension (incorporating also the $(0,1)$ last factor) and number of regions.
\begin{itemize}
\item[a.] $Z>\frac18$. $K(S)>0\ \forall S$. Have $v_S^\pm\ \forall S$. One such $M_\a$ of dimension 7.
\item[b.] $Z=\frac18$. $K(\sx)=0$. Have $v_S^\pm\ \forall S$ except $v_\sx^+=v_\sx^-$. One such region of dimension 6.
\item[c.] $Z<\frac18$; all $b_k^2<\frac12Z$. $K(\sx)<0$. Have $v_S^\pm$ for $1\le|S|\le5$, $v_{\sx,k}^0\ \forall k$, and $v_{\emptyset,k}^\pm\ \forall k$. One such region of dimension 7.
\item[d.] $Z<\frac18$; one $b_k^2=\frac12Z$. $K(\sx)<0$ and $K(\sxk)=0$. Have $v_S^\pm$ for $1\le|S|\le5$ with $v_\sxk^+=v_\sxk^-$. Also $v_{\sx,i}^0$ and $v_{\emptyset,i}^\pm$ for $i\ne k$. Six such regions of dimension 6.
\item[e.] $Z<\frac18$; $b_k^2=b_\ell^2=\frac12Z$; all other $b_i=0$. $K(\sx)<0$, $K(\sxk)=K([\![6]\!]-\{\ell\})=0$. Have $v_S^\pm$ for $1\le|S|\le5$ with $v_\sxk^+=v_\sxk^-$ and $v_{[\![6]\!]-\{\ell\}}^+=v_{[\![6]\!]-\{\ell\}}^-$. Also $v_{\sx,i}^0$ and $v_{\emptyset,i}^\pm$ for $i\ne k,\ell$. Fifteen such regions of dimension 2.
\item[f.] $Z<\frac18$; one $b_k^2>\frac12Z$. $K(\sx)<K(\sxk)<0$. Have $v_S^\pm$ for $S\ne\sx,\sxk,\{k\},\emptyset$, also $v_{\sxk,\ell}^0$ and $v_{\{k\},\ell}^\pm$ for all $\ell\ne k$. Also $v_{\sx,i}^0$ and $v_{\emptyset,i}^\pm$ for $i\ne k$. Six such regions of dimension 7.\end{itemize}

Continuing to let $b_i=a_i-\frac14$, and ignoring the $n$th component, for arbitrary $n$ regions $M_\a$ are determined by, for all subsets $S$ of $[\![n-1]\!]$, whether
$$\tfrac{n+3-2|S|}{16}+\dsum_Sb_i^2-\dsum_{\St}b_i^2$$
is positive, negative, or zero.
Which vertices of each type occur for $\cR(P)$ for $P$ in a region $M_\a$ is determined by the signs of the various $K(S)$, which are unique to the region. The coordinates of the vertices of $\cR(P)$ of various types are continuous  functions of the coordinates $(a_1,\ldots,a_n)$ of $P$. Moreover, equivalence of vertices is also determined by the sets $S$ and which $K(S)=0$, which is determined by $M_\a$.
The maps $\theta_P$ of Theorem \ref{thm1} send the various vertices of $\cR(P_0)$ in Theorem \ref{verthm} to the corresponding vertices of $\cR(P)$. The same formulas apply, just to different values of $a_i$.

We need to ``know'' the faces of the polytopes $\cR(P)$. All of our vertices are obtained as intersections of at least $n$ of the hyperplanes which defined the polytope in Definition \ref{def}
together with $x_n=a_n\pm1$.
All we need to say about the faces is that a $j$-face, for $j<n$, is a $j$-dimensional subset  which is the intersection of the polytope with exactly $n-j$ independent bounding hyperplanes. The face equals the convex hull of the vertices which lie in that intersection of hyperplanes.
 This justifies  the claims of the first paragraph of Theorem \ref{thm1} for $\cD=(0,\frac12)^{n-1}\times(0,1)$.

 For $\cD=(\frac12,1)^{n-1}\times(0,1)$, all formulas are extremely similar to those for $(0,\frac12)^{n-1}\times(0,1)$, as noted near the end of Section \ref{sec2}. We combine them into $((0,\frac12)\cup(\frac12,1))^{n-1}\times(0,1)$, although we could consider them separately. For $((0,\frac12)\cup(\frac12,1))^{n-1}\times\{0\}$, the polygons $\cR(P)$ are essentially the same as those we have been considering. We separate it to avoid continuity problems due to the relation (\ref{rel}).

After permuting variables, a general $\cD$ can be written as $$\{0,\tfrac12\}^j\times((0,\tfrac12)\cup(\tfrac12,1))^{n-1-j}\times D_n.$$ For $P=(a_1,\ldots,a_n)$ in this $\cD$,
$$\cR(P)=\prod_{i=1}^j[a_i-\tfrac12,a_i+\tfrac12]\times \cR(a_{j+1},\ldots,a_n)$$
with $a_i-\tfrac12\sim a_i+\frac12$ and $\cR(a_{j+1},\ldots,a_n)$ as previously described. This can be seen by observing that the inequalities (\ref{*}) and (\ref{neg}) will not have the terms for $i\le j$ because the segment from $P$ to points of $C(P)$ as in (\ref{CP}) will involve no change in $x_i$. The domains $M_\a$ here will be $\{x\}\times M_\a'$, where $x\in\{0,\frac12\}^j$ and $M_\a'$ is a domain that works for $\cR(a_{j+1},\ldots,a_n)$.

 We now explain why part (a) of Theorem \ref{thm1} is true. Suppose a sequence $P_i\in M_\a\subset\cD$ converges to $P\in M_{\a'}\subset\cD'$.

 {\bf Case 1.} $\cD'=\cD$: Let $N=n-1$. Let $\cD=\{P=(b_1,\ldots,b_N):-\frac14<b_i<\frac14\}$. For $S\subset[\![N]\!]$, define $K_S:\cD\to\R$ by
$$K_S(P)=\tfrac{N+4-2|S|}{16}+\dsum_Sb_i^2-\dsum_{\St} b_i^2.$$
This equals $K(S)$ defined in (\ref{Kdef}).
Note that $S\subsetneqq S'$ implies $K_{S'}(P)<K_S(P)$. Let $V_S=\{P:K_S(P)=0\}$.

For a function $\a:\cP([\![N]\!])\to\{-1,0,1\}$ which satisfies that if $S\subsetneqq S'$, then $\a(S')\le\a(S)$ with $\a(S')<\a(S)$ if either is 0, let $$M_\a=\{P\in\cD:\on{sgn}(K_S(P))=\a(S)\ \forall S\}.$$
Note that $M_\a$ is an open subset of the variety $V(\a):=\ds\bigcap_{\a(S)=0}V_S$. Let $d=\dim(V(\a))=\dim(M_\a)$.
\begin{prop} If $\la P_i\ra$ in $M_\a$ approaches $P\in M_{\a'}$, then either $\a=\a'$ or $\dim(M_{\a'})<d$.\end{prop}
\begin{proof} Since $K_S(P)$ is a continuous function of $P$, if $\a(S)=0$, then $\a'(S)=0$, while if $\a(S)\ne0$, then $\a'(S)=0$ or $\a(S)$. Thus if $\a'\ne\a$, there must be a set $S_0$ with $\a'(S_0)=0$ and $\a(S_0)\ne0$. The points $(b_1,\ldots,b_N)$ in $V(\a)$ are those which satisfy a linear system in $b_i^2$, with a row for each $S\in\a^{-1}(0)$. This linear system can be row reduced to a matrix which (after permuting variables) expresses $b_1^2,\ldots,b_{N-d}^2$ in terms of $b_{N-d+1}^2,\ldots,b_N^2$. Use these reduced rows to eliminate the first $N-d$ variables from the equation $K_{S_0}(P)=0$. If the obtained equation involving $b_{N-d+1}^2,\ldots,b_N^2$ is not identically 0, then it shows that $\dim(V(\a'))<d$.
If it is identically 0, this says that any $P$ which satisfies $K_S(P)=0$ for all $S\in\a^{-1}(0)$ must satisfy $K_{S_0}(P)=0$. Thus all $P\in M_\a$ have $K_{S_0}(P)=0$, and so $\a(S_0)=0$, contradiction.
\end{proof}

 {\bf Case 2.} $\cD'$ is formed from $\cD$ by changing $(0,1)$ in the last factor to $\{0\}$. This case is easy, since the $M_\a$'s for $\cD'$ are obtained from those of $\cD$ by changing the last factor from $(0,1)$ to $\{0\}$. Our considerations of $M_\a$ have generally been just with regard to the first $n-1$ variables. Since all 0's of $M_\a$ hold for $M_{\a'}$, the dimension with respect to the first $n-1$ variables of $M_{\a'}$ is $\le$ that of $M_\a$, but its dimension is smaller due to the last factor.

{\bf Case 3.} $\cD'$ is formed from $\cD$ by changing one or more $(0,\frac12)$ factors to $\{0\}$. (Other cases similar.) For simplicity, let's say that it is just the $(n-1)^{\text{st}}$ factor of $(0,\frac12)$ which is changed to 0. We consider the subsets $S$ of $[\![n-2]\!]$ which determine vertices of polytopes for $P$ in $(0,\frac12)^{n-2}\times\{0\}\times (0,1)$, and via the function $K$, determine the regions $M_{\a'}$ of $(0,\frac12)^{n-2}\times\{0\}\times (0,1)$. The key point is that, since $\ds\lim_{a\to0}\Delta_a=0$, if a sequence $P_i\in M_\a$ converges to $P\in M_{\a'}$, then for any $S\subset[\![n-1]\!]$,
$$\lim_{P_i\to P}K(S)(P_i)=K(S\cap[\![n-2]\!])(P).$$
Since, for $S\subset[\![n-2]\!]$, $K(S)>K(S\cup\{n-1\})$, it cannot happen that two different $S\subset[\![n-1]\!]$ have $K(S)=0$ in $M_\a$ but yield only one such when intersected with $[\![n-2]\!]$. So, similarly to Case 1, $\cD'$ has as many restricting equations as does $\cD$, but in a domain one dimension smaller.

If instead $P$ is in a set equivalent to $M_{\a'}\subset\cD'$, then it is similar to Case 2 or 3 above with one or more factors changed to $\{1\}$ instead of $\{0\}$. Our $M_\a$'s have been defined to be in various $\cD\subset[0,1)^n$. For example, if $$\cD=\{0,\tfrac12\}\times\{0,\tfrac12\}\times((0,\tfrac12)\cup(\tfrac12,1))^{n-3}\times(0,1),$$
the $M_\a$'s in the subset with 0 in the first two coordinates admit homeomorphic and equivalent $M_{\a'}$'s with the first two coordinates changed to 1, and the argument in Case 2 above shows that if a sequence in some $M_\a$ approaches a point $P=(1,1,a_3,\ldots,a_n)$, then the dimension of the new $M_{\a'}$ containing $P$ is less than that of $M_\a$. Similarly, if $P$, the limit of a sequence in $M_\a$, has final component 1, then it lies in a set $M_{\a'}$ which is homeomorphic  (under a $x\leftrightarrow 1-x$ correspondence in other components) to one of the $M_{\a'}$'s already considered, and its dimension is less than that of $M_\a$ by the argument in Case 3 above.
This completes the proof of part (a) of Theorem \ref{thm1}.

The proof of part (b) of Theorem \ref{thm1} involves two issues: (i) that vertices of $\cR(P_i)$ converge to vertices of $\cR(P)$, and (ii) that faces approach faces, of equal or smaller dimension. (i) is trivial if $P\in M_\a$. If $P$ is in the same $\cD$-set as $M_\a$, but a different $M_\a$, then it must be the case that, for some $S$, either $K(S)(P_i)$ or $K(S-\{k\})(P_i)$ approaches 0 in (\ref{xi}) or $K(S)(P_i)$ or $K(S\cup\{k\})(P_i)$ approaches 1 in (\ref{trun}). By computing the limiting value of  $x_k$  in each case, we see that $v_{S,k}^0\to v_S^+=v_S^-$, $v_{S,k}^0\to v_{S-\{k\}}^+=v_{S-\{k\}}^-$, $v_{S,k}^\pm\to v_S^\pm$, or $v_{S,k}^\pm\to v_{S-\{k\}}^\pm$ in the four cases, so a ``standard'' vertex is obtained in the limit. If $P_i$ approaches a point with some component equal to 0 or $\frac12$, then, in the formula $K(S)=\frac12-\Delta(S)+\Delta(\St)$, the limit as $a_i$ approaches 0 or $\frac12$ equals the value it would have if the $i$th component of $S$ was omitted, so a vertex is obtained in the limit.

Regarding (ii): Vertices in a $j$-face $F$ of $\cR(P_0)$ satisfy $n-j$ independent equalities in Definition \ref{def}, perhaps including also $x_n=a_n\pm1$ or $x_n=0$. The corresponding vertices of $\cR(P_i)$ satisfy the same equalities. For example, if the vertices of $F$ satisfy $x_n=\frac12+\sum 2a_j(P_0)x_j$, then vertices of $\cR(P_i)$ satisfy $x_n=\frac12+\sum2a_j(P_i)x_j$ by the definition of $\theta_{P_i}$, and since $a_j(P_i)$ approaches $a_j(P)$, the same is true in $\cR(P)$. So the vertices of $\cR(P)$ satisfy at least $n-j$ independent equalities of the appropriate form, and hence span a face of dimension $\le j$.

We amplify on the word ``independent'' by showing that the rank of a system of such equations at the limit point cannot be less than that for the points in the sequence, by the nature of the equations. After inserting values $x_i=a_i-\frac12+\eps_i$ of wall variables, and then subtracting one equation (associated to (\ref{one})) from all the others, the system is effectively a matrix of 0's and 1's determined by the differences of the values of $\delta_i$'s from (\ref{one}) in the equations being subtracted. These do not change under taking limits. The system at the limit point could involve more equations than those in the sequence, for example if a sequence of points on a slant hyperplane approach a point on a wall, and this would decrease the dimension of the face.

\section{Lower bounds}

\begin{thm}
The geodesic complexity of the higher Klein bottle $K_n$ satisfies
\[ \GC(K_n) \ge 2n.\]
\end{thm}

\begin{proof}

Assume that there exists a decomposition $X \times X = \bigsqcup_{i=0}^k E_i$ satisfying Definition \ref{gcdef}. We need to show that $k \ge 2n+1$. We will do this in three steps.

\vspace{.15cm}

\textbf{\underline{First part: $k \ge n+1$}}

\vspace{.15cm}

Fix a point $P$ in the universal cover of $K_n$. 
Let $\mathcal{R}$ be the polytope associated to $P$. Fix a vertex $V$ of $\mathcal{R}$. Let $U_\epsilon$ be a ball of (sufficiently small) radius $\epsilon$ around the projection of $V$ in $K_n$.

The ball $U_\epsilon$ is homeomorphic to a ball in $\R^n$ centered around the origin. This ball is divided into open chambers, which correspond to the projection of the interior of $\mathcal{R}$. For example, if $V=V_+$ in Figure \ref{newfig}, which is a rendering of Figure \ref{expl}, the chambers are projections of the small sectors at $V_+$, $C_{--}$, and $C_{+-}$ in the interior of the  hexagon. The diagram also indicates their projection in $K_2$. The chambers are divided by $(n-1)$-walls which correspond to $(n-1)$-dimensional half spaces going through the origin in $\R^n$. The $(n-1)$-walls are the projection of part of the $(n-1)$-skeleton of $\mathcal{R}$. Those walls intersect along lower dimensional walls, which are the projections of the skeletons of the corresponding dimension.

\begin{minipage}{6in}
\begin{fig}\label{newfig}
{\bf Chambers}
\begin{center}
\begin{\tz}[scale=.45]
\draw (0,0) -- (4,0) -- (4,4) -- (0,4) -- (0,0);
\draw [red] (1,3.5) -- (3,2.5) -- (3,-.5) -- (1,-1.5) -- (-1,-.5) -- (-1,2.5) -- (1,3.5);
\node at (1,1)  [circle,fill,inner sep=1pt]{};
\node at (3.8,2.5) {$C_{++}$};
\node at (1,3.7) {$V_+$};
\node at (1,-1.8) {$V_-$};
\node at (-1.8,2.5) {$C_{-+}$};
\node at (-1.8,-.5) {$C_{--}$};
\node at (3.8,-.5) {$C_{+-}$};
\node at (5.5,1) {$\mapright{p}$};
\draw (8,1) -- (8,2.5);
\draw (6.7,.25) -- (8,1) -- (9.3,.25);
\draw [domain=210:330] plot ({1+.6* cos(\x)},{3.5+.6*sin(\x)});
\draw [domain=90:210] plot ({3+.6* cos(\x)},{-.5+.6*sin(\x)});
\draw [domain=-30:90] plot ({-1+.6* cos(\x)},{-.5+.6*sin(\x)});
\end{\tz}
\end{center}
\end{fig}
\end{minipage}

Choose a point in an open chamber. Let $Q$ denote its unique preimage in the interior of the polytope $\mathcal{R}$. There is a unique shortest path between $P$ and $Q$, and its projection to $K_n$ is the unique shortest path between $p(P)$ and $p(Q)$.

Now choose a point $x$ in the interior of an $(n-1)$-wall. Every neighborhood of $x$ intersects the two chambers meeting at this half-space. Suppose we choose two sequences $\la x^1_i\ra$ and $\la x^2_i\ra$ converging to $x$, each sequence contained in a different chamber. For every  $x^1_i$ there is a unique shortest path over $(p(P),x^1_i)$, and analogously for $x^2_i$. However, the shortest paths over $(p(P),x^1_i)$ and over $(p(P),x^2_i)$ converge to two different shortest paths over $(p(P),x)$.  This is because approaching $x$ from each chamber corresponds to approaching two distinct points in the polytope in the universal cover, which are the two preimages of $x$.

Assume, for the sake of contradiction, that there is a neighborhood $W_\epsilon$ of $x$ such that $\{p(P)\}\times W_\epsilon$ is contained in $E_1$. Then we may assume that $E_1$ contains both sequences $(p(P),x^1_i)$ and $(p(P),x^2_i)$, as well as $(p(P),x)$ itself. By the preceding  paragraph, this implies that there cannot exist a continuous choice of shortest path over $E_1$. However, a local section is required to exist over each $E_i$. This shows that every neighborhood of $(p(P),x)$ intersects nontrivially at least two different sets $E_i$. Equivalently, $(p(P),x)$ lies in the closure of at least two $E_i$.

Fix the point $P$ throughout this proof. Denote $\tilde E_i = E_i \cap (\{P\}\times K_n)$ and omit the first coordinate from the notation.

The rest of the argument proceeds by induction: Assume that every point $y$ which is in the interior of an $(n-k+1)$-wall lies in the closure of at least $k$ many $\tilde E_i$. We want to show that every point $x$ in the interior of an $(n-k)$-wall lies in the closure of at least $k+1$ many $\tilde E_i$. Choose a sufficiently small neighborhood $W_\epsilon$ of $x$.  Assume, for the sake of contradiction, that $W_\epsilon$ is contained in $\tilde E_1 \sqcup \tilde E_2 \sqcup \dots \sqcup \tilde E_k$. Furthermore assume that $x$ is in $\tilde E_1$.

For the following argument it is useful to picture a tiling of the universal covering space $\R^n$ by copies of the polytope $\mathcal{R}$, according to the deck transformations (translations and reflections). From this point of view, the chambers can be thought of as the projections of (pieces of)  those polytopes which are adjacent to the vertex $V$.

Let $L$ be the intersection of $W_\epsilon$ with an $(n-k+1)$-wall which contains $x$ in its closure. The point $x$ is in the closure of several $(n-k+1)$-walls and no chamber contains all of those $(n-k+1)$-walls in its boundary.

 Choose a sequence $\la y_i\ra$ in $L$ converging to $x$. Since  $L$ is contained in $\tilde E_1 \sqcup \tilde E_2 \sqcup \dots \sqcup \tilde E_k$, by the induction hypothesis each $y_i$ lies in the closure of $\tilde E_1$ (as well as $\tilde E_2,\ldots,\tilde E_k$). Therefore, there exists a sequence in $\tilde E_1$ converging to $y_i$, which might be constant. By a diagonal argument we can find a sequence $\la x_i\ra$ converging to $x$ such that the entire sequence is either contained in $L$ or contained in the interior of one chamber which is adjacent to $L$.

Now consider the shortest path $s_1(p(P),x)$ given by the local section over $E_1$. By continuity, $s_1(p(P),x)$ needs to coincide with the limit of $s_1(p(P),x_i)$, as $i$ goes to infinity. This implies that the path $s_1(p(P),x)$ needs to intersect a chamber which is adjacent to $L$, by the previous paragraph. Note that if the $x_i$'s were all in $L$, then a subsequence of them has $s_1(p(P),x_i)$ intersecting a single chamber adjacent to $L$.

However, we could repeat the argument with any other $(n-k+1)$-wall which has $x$ in its closure, reaching an analogous conclusion. As we noted, those $(n-k+1)$-walls can be chosen so that they do not have a chamber in common, which yields a contradiction.

This concludes the induction argument. Because $V$ is a vertex (or 0-wall) it follows that the pair $(p(P),p(V))$ lies in the closure of at least $n+1$ many $E_i$.

\vspace{.15cm}

\textbf{\underline{Second part $k \ge 2n$}}

\vspace{.15cm}

Consider a point $P = (a_1,a_2, \dots ,a_n)$ in the universal cover of $K_n$ with coordinates $0 < a_i < {\frac12}$ for $1 \le i \le n-2$, $a_{n-1} = {\frac12}$ and $a_n=0$. Assume that all the $a_i$ are very close to ${\frac12}$ such that there are no middle or truncating vertices or unusual equivalences between vertices in $\cR(P)$.

Let $S$ be the labeled subset of $[\![n-2]\!]$ associated to some vertex $V_0$ of $\mathcal{R}(P)$ and assume that $V_0$ has coordinate $x_{n-1}=0$.

Now fix all coordinates as above except that $0 < a_{n-1} <{\frac12}$. As $a_{n-1}$ approaches ${\frac12}$ from the left, the polytopes of the corresponding points have two vertices approaching the vertex $V_0$. Those vertices correspond to the labeled sets $S$ and $S \cup \{ n-1 \}$ as subsets of $[\![n-1]\!]$ (with $\eps_{n-1}=0$). We will call them $V_1$ and $V_2$ respectively. Keep in mind that $V_1$ and $V_2$ vary as $a_{n-1}$ changes.

Concretely, $V_1$ and $V_2$ only differ at $x_{n-1}$ and $x_n$. The vertex $V_1$ has $x_{n-1} = {\frac12} - a_{n-1}$ and the vertex $V_2$ has coordinate $x_{n-1} = a_{n-1} - \frac12$. As $a_{n-1}$ approaches $\frac12$, both $x_{n-1}$ coordinates approach 0, which is precisely the $x_{n-1}$ coordinate of $V_0$. The $x_n$ coordinates of $V_1$ and $V_2$ differ by $2 \Delta_{a_{n-1}}$ which goes to 0 as $a_{n-1}$ goes to $\frac12$, and the other $\Delta_{a_i}$ summands in (\ref{Kdef}) are equal in $V_1$, $V_2$, and $V_0$.

How many ways are there to make a continuous choice of shortest paths between the projections of the $P = (a_1,a_2, \dots ,a_n)$ and the projections of those three types of vertices? It boils down to choosing paths between the points $P$ and other vertices of $\mathcal{R}(P)$ which are equivalent to $V_1$, $V_2$ and $V_0$ respectively.

We assumed that all vertices are regular and have a small $x_n$ coordinate. By Proposition \ref{equivthm}, such vertices are equivalent to all vertices which correspond either to the same labeled set $S$ and the same $x_n$ or to the complementary labeled set $\St$ and inverted $x_n$.

While there are several consistent choices of shortest paths over the three equivalence classes, all those paths need to go in negative direction in the $n-1$ coordinate. This is because in order for two representatives of the equivalence classes of $V_1$ and $V_2$ to approach each other, one of the two vertices needs to have $x_{n-1}=\frac12- a_{n-1}$ and the other one $x_{n-1} = a_{n-1} - \frac12$. This is illustrated in the case $n=2$ in the left half of Figure \ref{fig:lowerbound}.

\begin{fig}

    \label{fig:lowerbound}
{As $a_1$ approaches $\frac12$, and $V_1$ ($V_1'$) and $V_2$ ($V_2'$) approach the vertex $V_0$ ($V_0'$), the only continuous choice of minimal geodesics requires all geodesics go the left (right) when $a_1$ is approaching $\frac12$ from the left (right).}
  \begin{center}
    \includegraphics[scale=2]{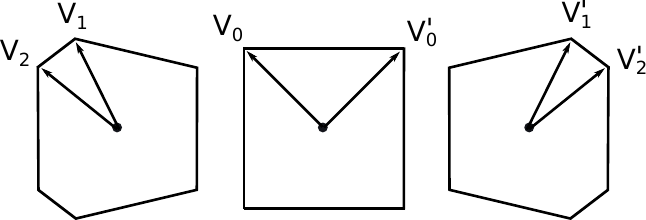}

    \end{center}
  \end{fig}

Repeating the whole procedure except with $\frac12< a_{n-1} < 1$ approaching $\frac12$ from the right instead is completely analogous. However, the conclusion in that case is completely opposite: in all possible continuous choices of shortest paths over the three types of vertices, the paths need to go in positive direction in the $n-1$ coordinate. This is because in that case the coordinates are $x_{n-1} = 1 - (a_{n-1} - \frac12) $ and $x_{n-1} = a_{n-1} + \frac12$.

Now assume, for the sake of contradiction, that there is a ball $U_\epsilon$ around $(p(P_0),V_0)$ which is contained in $E_1 \sqcup \dots \sqcup E_{n+1}$ and that $(p(P_0),V_0)$ lies in $E_1$. In the first part of the proof we showed that at least $n+1$ sets $E_i$ accumulate at every pair of the form $(p(P),p(V))$. In particular, $E_1$ accumulates at every pair $(p(P),p(V))$ contained in $U_\epsilon$.

If $\epsilon$ is small enough, the paths in the image of the section over $E_1\cap U_\epsilon$ need to be very close to the path $s_1(p(P_0),p(V_0))$. In particular, they need to go in the same direction in the $n-1$ coordinate. Because $E_1$ accumulates at all the pairs $(p(P),p(V_1))$ and $(p(P),p(V_2))$ in $U_\epsilon$, this means that there are shortest paths with endpoints very close to $p(V_1)$ and $p(V_2)$ which are close to $s_1(p(P_0),p(V_0))$ and thus close to each other. As we saw above, if two shortest paths from $p(P)$ to $p(V_1)$ and $p(V_2)$ respectively are close to each other they need to go in negative $n-1$ direction if $0 < a_{n-1} <\frac12$ and in positive $n-1$ direction if $\frac12 < a_{n-1} <1$. This implies that the paths cannot all go in the same direction.


%
%
%
%

This yields a contradiction. Therefore, at least $n+2$ many $E_i$ intersect every neighborhood of every pair of the form $(p(P),p(V))$ with $0 < a_i < \frac12$ for $1 \le i \le n-2$ and $a_{n-1} = \frac12$. In other words, at least $n+2$ sets $E_i$ accumulate at every pair of that form.

Note that the assumption $0 < a_i < \frac12$ can be relaxed to $0 < a_i < \frac12$ for some $1\le i \le n-2$ and $\frac12 < a_i < 1$ for other $1\le i \le n-2$, and the argument above still goes through.

The rest of the argument proceeds by induction. The next induction step is to let the coordinate $a_{n-2}$ approach $\frac12$ from the left and the right and show that yet another set $E_i$ is needed in every neighborhood of the corresponding vertex.

After letting all coordinates $a_i$ with $1\le i \le n-1$ approach $\frac12$, one by one, we conclude that every vertex of the polytope $\mathcal{R}(P)$ for $P = (\frac12, \dots , \frac12 , 0)$ is in the closure of $(n+1) + (n-1) = 2n$ many $E_i$.

%
%
%
%
%
%
%
%
%
%
%
%

\vspace{.15cm}

\textbf{\underline{Third part $k \ge 2n+1$}}

\vspace{.15cm}

Let $S_{2n}$ denote the set of all pairs $(p(\frac12, \dots , \frac12 ,  a_n),p(1, \dots , 1 ,  a_n + \frac12))$. We just proved that every such pair of points lies in the closure of $2n$ sets $E_i$.

Assume that $k = 2n$, i.e., there are exactly $2n$ sets $E_i$ in the decomposition of $K_n \times K_n$. Then $S_{2n}$ is contained in the closure of every $E_i$.

Because the $E_i$ are ENRs, they have to be locally closed, which means that they are open in their closure. In particular all the intersections $E_i\cap S_{2n}$ need to be open in $S_{2n}$, as well as disjoint. Because $S_{2n}$ is connected, this implies that if $E_i \cap S_{2n} \neq \emptyset$, then $S_{2n} \subset E_i$. Therefore there is a continuous choice of geodesics over $S_{2n}$. However, to make a continuous choice of path for all pairs $(p(\frac12, \dots , \frac12 ,  a_n),p(1, \dots , 1 ,  a_n + \frac12))$ it is necessary to make a consistent choice of left or right in the first coordinate, say. This is of course not possible for all $a_n$, because as $a_n$ goes from 0 to 1, the directions switch.

\end{proof}

\end{document}